\documentclass[11pt,a4paper]{article}
\usepackage{amsmath,amsthm,amssymb,url,longtable,changes,enumitem}
\usepackage{multirow}
\usepackage{tikz-cd}

\theoremstyle{plain}
\newtheorem{theorem}{Theorem}[section]
\newtheorem{conjecture}[theorem]{Conjecture}

\newtheorem{lemma}[theorem]{Lemma}

\newtheorem*{claim*}{Claim}
\newtheorem*{problem*}{Problem}
\newtheorem*{conjecture*}{Conjecture}
\newtheorem*{theorem*}{Theorem}

\theoremstyle{definition}
\newtheorem{definition}[theorem]{Definition}

\newcommand{\ad}{\mathrm{ad}}

\newcommand{\Aut}{\mathrm{Aut}}

\newcommand{\Spec}{\mathrm{Spec}}

\newcommand{\new}{\mathrm{new}}

\newcommand{\stab}{\mathrm{Stab}}

\newcommand{\A}{\mathrm{A}}
\newcommand{\B}{\mathrm{B}}
\newcommand{\C}{\mathrm{C}}
\renewcommand{\L}{\mathrm{L}}

\renewcommand{\phi}{\varphi}
\renewcommand{\epsilon}{\varepsilon}
\begin{document}

\title{An expansion algorithm for constructing axial algebras}

\author{Justin M\textsuperscript{c}Inroy\footnote{School of Mathematics, 
University of Bristol, Bristol, BS8 1TW, UK, and the Heilbronn Institute 
for Mathematical Research, Bristol, UK, email: 
justin.mcinroy@bristol.ac.uk} \and Sergey Shpectorov\footnote{School of 
Mathematics, University of Birmingham, Edgbaston, Birmingham, B15 2TT, 
UK, email: S.Shpectorov@bham.ac.uk}}


\date{}
\maketitle

\begin{abstract}
An \emph{axial algebra} $A$ is a commutative non-associative algebra 
generated by primitive idempotents, called \emph{axes}, whose adjoint 
action on $A$ is semisimple and multiplication of eigenvectors is 
controlled by a certain fusion law. Different fusion laws define 
different classes of axial algebras. 

Axial algebras are inherently related to groups. Namely, when the fusion 
law is graded by an abelian group $T$, every axis $a$ leads to a subgroup 
of automorphisms $T_a$ of $A$. The group generated by all $T_a$ is 
called the \emph{Miyamoto group} of the algebra. We describe a new 
algorithm for constructing axial algebras with a given Miyamoto group. A 
key feature of the algorithm is the expansion step, which allows us to 
overcome the $2$-closeness restriction of Seress's algorithm computing 
Majorana algebras.

At the end we provide a list of examples for the Monster fusion law, 
computed using a {\sc magma} implementation of our algorithm. 
\end{abstract}

\section{Introduction}
Axial algebras are a new class of non-associative algebras introduced 
recently by Hall, Rehren and Shpectorov \cite{Axial1} as a broad 
generalization of the class of Majorana algebras of Ivanov \cite{I09}. 
The key features of these algebras came from the theory of vertex 
operator algebras (VOAs) which first arose in connection with 2D 
conformal field theory and they were used by Frenkel, Lepowsky and 
Meurman \cite{FLM} in their construction of the moonshine VOA 
$V^\natural$ whose automorphism group is the Monster $M$, the largest 
sporadic finite simple group. The rigorous theory of VOAs was developed 
by Borcherds \cite{B86} as part of his proof of the monstrous moonshine 
conjecture.  

Roughly speaking, VOAs are infinite dimensional graded vector spaces 
$V = \bigoplus_{i=0}^\infty V_i$ with infinitely many products linked in 
an intricate way. The Monster was originally constructed by Griess 
\cite{G82} as the automorphism group of a $196,883$-dimensional 
non-associative real algebra, called the Griess algebra, and the 
Moonshine VOA $V^\natural$ contains a unital deformation of the Griess 
algebra as its weight $2$ part $V_2^\natural$.

One of the key properties that axial algebras axiomatise was first 
observed in VOAs by Miyamoto \cite{Miy96}. He showed that you could 
associate involutory automorphisms $\tau_a$ of a VOA $V$, called 
\emph{Miyamoto involutions}, to special conformal vectors $a$ in $V_2$ 
called \emph{Ising vectors} \cite{Miy96}.  Moreover, in the Moonshine 
VOA, $\frac{a}{2}$ is an idempotent in the Griess algebra $V_2^\natural$, 
called a \emph{$2\mathrm{A}$-axis} because the corresponding involution 
$\tau_a$ lies in the class $2A$ of the Monster $M$.

The subalgebras of the Griess algebra generated by two 
$2\mathrm{A}$-axes, which we call \emph{dihedral subalgebras}, were 
first studied by Norton \cite{C85}. He showed that the isomorphism class 
of the dihedral subalgebra generated by $2A$-axes $a$ and $b$ is 
determined by the conjugacy class of the product $\tau_a \tau_b$. There 
are nine classes in $M$ containing products of two $2A$ involutions, 
labelled $1\textup{A}$, $2\textup{A}$, $2\textup{B}$, $3\textup{A}$, 
$3\textup{C}$, $4\textup{A}$, $4\textup{B}$, $5\textup{A}$ and 
$6\textup{A}$. Remarkably, Sakuma \cite{sakuma} showed that each sub VOA 
generated by two Ising vectors is also one of nine isomorphism types. 
Therefore, the above nine classes in $M$ are used as labels for the 
$2$-generated VOAs arising in Sakuma's theorem. 

Sakuma's result was extended to Majorana algebras in \cite{IPSS} and 
later to axial algebras with the Monster fusion law and a Frobenius 
form\footnote{Franchi, Mainardis and Shpectorov announced at the Axial 
Algebra Focused Workshop in Bristol in May 2018 that the Frobenius form 
condition has been removed.} in \cite{Axial1}.

Majorana algebras were introduced by Ivanov \cite{I09} to abstract the 
properties of $2A$-axes. Axial algebras provide a further broad 
generalisation removing the less essential restrictions of Majorana 
algebras. An \emph{axial algebra} is a commutative non-associative 
algebra generated by \emph{axes}, that is, primitive semisimple 
idempotents whose adjoint eigenvectors multiply according to a certain 
fusion law.  We say that an axial algebra is of \emph{Monster type} if 
its fusion law is the Monster fusion law (see Table 
\ref{tab:monsterfusion}).  For the exact details see Section 
\ref{sec:background}.  A Majorana algebra is then an axial algebra of Monster type which satisfies some additional conditions.  

Whenever the fusion law is $T$-graded, where $T$ is an abelian group, 
associated to each axis $a$ we get an automorphism $\tau_a(\chi)$ for 
every linear character $\chi\in T^*$. We define 
$T_a=\langle\tau_a(\chi):\chi\in T^*\rangle$, which has size at most 
$|T|$. The group generated by the $T_a$ for all axes $a$ is called the 
\emph{Miyamoto group}. For the important motivating example of the 
Griess algebra, the fusion law is $\mathbb{Z}_2$-graded and so, for 
every axis $a$, there is an involutory automorphism 
$\tau_a:=\tau_a(\chi_{-1})$ corresponding to the unique non-trivial 
character $\chi_{-1}$ of $\mathbb{Z}_2$. The Miyamoto group generated 
by all the $\tau_a$ is the Monster $M$ and the $\tau_a$ are the 
whole $2A$ conjugacy class.

Another example of a class of axial algebras with a different fusion law are algebras of Jordan 
type comprising Matsuo algebras, whose Miyamoto groups are 
$3$-transposition groups, and Jordan algebras, whose Miyamoto groups 
include classical groups and groups of exceptional Lie type $F_4$ and 
$G_2$. 

\begin{problem*}
For a given fusion law, which groups $G$ occur as the Miyamoto group of an axial algebra?
\end{problem*}

Seress \cite{seress} addressed this question for the class of 
Majorana algebras by developing an algorithm that computes, for a given 
$6$-transposition group $G$, possible $2$-closed Majorana algebras. (An 
axial algebra is $2$-closed if it is spanned by axes and by products of 
two axes.) He also provided a GAP implementation of his algorithm. 
However, his code was lost when he sadly died. Pfeiffer and Whybrow 
\cite{Maddycode} have recently developed a new and improved GAP 
implementation of Seress' algorithm.

%
%

In this paper we describe a new algorithm for addressing the second of 
the above questions and present results obtained using a {\sc magma} 
implementation \cite{ParAxlAlg, magma} of this algorithm. The new 
algorithm differs from Seress's algorithm in several key ways. Our 
algorithm works for a general axial algebra over an arbitrary field with 
any $T$-graded fusion law, rather than just for the Monster fusion law 
(which is $\mathbb{Z}_2$-graded) over $\mathbb{R}$. Crucially, we do not 
assume that the algebra is $2$-closed. Indeed we find quite a few 
examples that are not $2$-closed. We also do not assume that the algebra 
has an associating bilinear form (a \emph{Frobenius form}), whereas 
Seress assumes this and also that the form is positive definite. We do 
not assume the so-called 2Aa, 2Ab, 3A, 4A, 5A conditions (see 
\cite[page 314]{seress}) which restrict the configuration of the 
dihedral subalgebras. Finally, we do not require that the axes $a$ be in bijection with the axis subgroups $T_a$.

Let $\mathcal{F}$ be a $T$-graded fusion law and $G$ be a group acting 
on a set $X$.  We aim to build an axial algebra where the action on the 
axes by (a supergroup of) the Miyamoto group is given by the action of 
$G$ on $X$. In Section \ref{sec:shape} we rigorously define admissible 
$\tau$-maps and the shape of an algebra. Roughly speaking, 
$\tau\colon X\times T^*\to G_0\leq G$ is an \emph{admissible $\tau$-map} 
if it has the properties that the map $(a,\chi)\mapsto\tau_a(\chi)$ in 
an axial algebra has. The subgroup $G_0\unlhd G$ generated by the image 
of this map will be our Miyamoto group. The \emph{shape} is a choice 
of $2$-generated subalgebra for each pair of axes $a,b\in X$. Since the 
isomorphism class of $2$-generated subalgebras is preserved under 
automorphisms, in particular, under the action of the Miyamoto group, 
we need only make one choice for each conjugacy class of pairs of axes. 
In fact, there are some addition constraints on the shape given by 
containment of $2$-generated subalgebras in one another as described 
in Section \ref{sec:shape}. Our algorithm takes $\mathcal{F}$, $G$, $X$, 
$\tau$ and the shape as its input.  We show the following:

\begin{theorem*}
Suppose that the algorithm terminates and returns $A$. Then $A$ is a 
\textup{(}not necessarily primitive\textup{)} axial algebra generated 
by axes $X$ with Miyamoto group $G_0$, $\tau$-map $\tau$ and of the 
given shape.

Moreover, the algebra $A$ is universal. That is, given any other axial 
algebra $B$ with the same axes $X$, Miyamoto group $G_0$, $\tau$-map 
$\tau$ and shape, $B$ is a quotient of $A$.
\end{theorem*}

We find several new examples of axial algebra with the Monster fusion 
law. Some of these are $3$-closed examples (in fact we find some 
examples which are $5$-closed), but we also find many examples that do 
not satisfy the so-called M8-condition. This condition severely 
restricts the allowable intersections of certain dihedral subalgebras 
in the shape. We also see in our results several shapes which do not 
satisfy the 2Aa, 2Ab, 3A, 4A, 5A conditions (see 
Section \ref{sec:dihedral}), but still produce good axial 
algebras.

Interestingly, all the algebras we construct have a Frobenius form 
which is non-zero on the axes and invariant under the action of the 
Miyamoto group, even though we do not require this in our algorithm. 
Moreover, in all our examples, the form is positive semi-definite. It 
is known that axial algebras of Jordan type (those with three 
eigenvalues, $1$, $0$ and $\eta$) all have Frobenius forms and it 
has previously been observed that the other known examples also have 
Frobenius forms. Such a form, if it does exist, is uniquely 
determined by its values on the axes. So we make the following 
conjecture.

\begin{conjecture*}
All primitive axial algebras of Monster type admit a Frobenius form 
which is non-zero on the axes and invariant under the action of the 
Miyamoto group.
\end{conjecture*}

The structure of the paper is as follows. In Section 
\ref{sec:background}, we define axial algebras and discuss various 
properties such as Miyamoto involutions and dihedral subalgebras.  
We define the shape of an algebra in Section \ref{sec:shape}. 
Section \ref{sec:preliminaries} gives some lemmas and further 
properties of axial algebras which we will need. Our main result 
is the algorithm which is described in Section \ref{sec:algorithm}. 
Finally, in Section \ref{sec:results}, we present examples computed 
by our {\sc magma} implementation of the algorithm.

\medskip

We thank Simon Peacock for some useful comments on an early draft 
of this paper.

\section{Background}\label{sec:background}

We will review the definition and some properties of axial algebras 
which were first introduced by Hall, Rehren and Shpectorov in 
\cite{Axial1}. We will pay particular attention to the motivating 
examples coming from the Monster sporadic finite simple group and 
also indicate the extra conditions for such an axial algebra to be 
a Majorana algebra.

\begin{definition}
Let $\mathbb{F}$ be a field, $\mathcal{F}\subseteq\mathbb{F}$ a 
subset, and 
$\star\colon\mathcal{F}\times\mathcal{F}\to 2^{\mathcal{F}}$ a 
symmetric binary operation. We call the pair 
$(\mathcal{F},\star)$ a \emph{fusion law over $\mathbb{F}$}. A 
single instance $\lambda\star\mu$ is called a \emph{fusion rule}.
\end{definition}

Abusing notation, we will often just write $\mathcal{F}$ for 
$(\mathcal{F},\star)$. We can also extend the operation $\star$ to 
subsets $I,J\subseteq\mathcal{F}$ in the obvious way: $I\star J$ is the 
union of all $\mu\star\nu$ for $\mu\in I$ and $\nu\in J$. We note that 
after extending the operation, $(2^\mathcal{F},\star)$ is closed 
and so is a commutative magma. We will further abuse notation and mix 
subsets and elements.

Let $A$ be a commutative non-associative (i.e.\ not-necessarily-associative) algebra over $\mathbb{F}$. For an element $a\in A$, 
the \emph{adjoint endomorphism} $\ad_a\colon A\to A$ is defined by 
$\ad_a(v):=av$, for all $v \in A$. Let $\Spec(a)$ be the set of 
eigenvalues of $\ad_a$, and for $\lambda\in\Spec(a)$, let 
$A_\lambda^a$ be the $\lambda$-eigenspace of $\ad_a$. Where the 
context is clear, we will write $A_\lambda$ for $A_\lambda^a$.  We 
will also adopt the convention that for subsets 
$I\subseteq\mathcal{F}$, $A_I:=\bigoplus_{\lambda\in I}A_\lambda$.

\begin{definition}\label{axialalgebra}
Let $(\mathcal{F},\star)$ be a fusion law over $\mathbb{F}$. An 
element $a\in A$ is an \emph{$\mathcal{F}$-axis} if the following 
hold:
\begin{enumerate}
\item $a$ is \emph{idempotent} (i.e.\ $a^2=a$);
\item $a$ is \emph{semisimple} (i.e.\ the adjoint $\ad_a$ is 
diagonalisable);
\item $\Spec(a)\subseteq\mathcal{F}$ and 
$A_\lambda A_\mu\subseteq A_{\lambda\star\mu}$ for all 
$\lambda,\mu\in\Spec(a)$.
\end{enumerate}
Furthermore, we say that the $\mathbb{F}$-axis $a$ is 
\emph{primitive} if $A_1=\langle a\rangle$.
\end{definition}

Note that, when $\Spec(a)\neq\mathcal{F}$, we can still talk of 
$A_\lambda^a$ for all $\lambda\in\mathcal{F}$: if 
$\lambda\notin\Spec(a)$ then $A_\lambda^a=0$. With this 
understanding, the last condition means that 
$A_\lambda A_\mu\subseteq A_{\lambda\star\mu}$ for all 
$\lambda,\mu\in\mathcal{F}$.

\begin{definition}
An \emph{$\mathcal{F}$-axial algebra} is a pair $(A,X)$ such that 
$A$ is a commutative non-associative algebra and $X$ is a set of 
$\mathcal{F}$-axes generating $A$. An axial algebra is 
\emph{primitive} if it is generated by primitive axes.
\end{definition}

Where the fusion law is clear from context, we will drop 
the $\mathcal{F}$ and simply use the term \emph{axis} and 
\emph{axial algebra}. Although an axial algebra has a distinguished generating set $X$, 
we will abuse the above notation and just write $A$ for the pair 
$(A,X)$.  Note that it has been usual in the literature 
to drop the adjective primitive and consider only primitive axial 
algebras.

The fusion law over $\mathbb{R}$ associated to the Monster is given 
by Table \ref{tab:monsterfusion}.
\begin{table}[!htb]
\setlength{\tabcolsep}{4pt}
\renewcommand{\arraystretch}{1.5}
\centering
\begin{tabular}{c||c|c|c|c}
 & $1$ & $0$ & $\frac{1}{4}$ & $\frac{1}{32}$ \\ \hline \hline
$1$ & $1$ &  & $\frac{1}{4}$ & $\frac{1}{32}$ \\ \hline
$0$ &  & $0$ &$\frac{1}{4}$ & $\frac{1}{32}$ \\ \hline
$\frac{1}{4}$ & $\frac{1}{4}$ & $\frac{1}{4}$ & $1, 0$ & $\frac{1}{32}$ \\ \hline
$\frac{1}{32}$ & $\frac{1}{32}$  & $\frac{1}{32}$ & $\frac{1}{32}$ & $1, 0, \frac{1}{4}$ 
\end{tabular}
\caption{Monster fusion law}\label{tab:monsterfusion}
\end{table}
This fusion law is exhibited by the so-called $2A$-axes in the Griess 
algebra. Indeed, noting that these generate the Griess algebra shows 
that it is an axial algebra. We say that an axial algebra is of 
\emph{Monster type} if it is an axial algebra with the Monster fusion 
law.

By definition, an axial algebra $A$ is spanned by products of the 
axes. We say that $A$ is \emph{$m$-closed} if $A$ is spanned by 
products of length at most $m$ in the axes.

\begin{definition}
A \emph{Frobenius form} on an axial algebra $A$ is a non-zero 
(symmetric) bilinear form 
$(\cdot,\cdot)\colon A\times A\to\mathbb{F}$ such that the form 
associates with the algebra product. That is, for all $x,y,z\in A$,
\[
(x,yz)=(xy,z).
\]
\end{definition}

We will be particularly interested in Frobenius forms such that 
$(a,a)\neq 0$, for all $a\in X$. That is, they are non-zero on the set 
of axes $X$. Note that an associating bilinear form on an axial algebra 
is necessarily symmetric \cite[Proposition 3.5]{Axial1}. Also, the 
eigenspaces for an axis in an axial algebra are perpendicular with 
respect to the Frobenius form.

\begin{lemma}\label{formunique}\textup{\cite[Lemma 4.17]{axialstructure}}
Suppose that $A$ is a primitive axial algebra admitting a Frobenius 
form. Then the form is uniquely determined by the values $(a,a)$ on the 
axes $a\in X$.
\end{lemma}

Majorana algebras were introduced by Ivanov by generalising certain 
properties found in subalgebras of the Griess algebra \cite{I09}. Axial 
algebras were developed as a generalisation of Majorana algebras, so 
Majorana algebras can be thought of as the precursor of axial algebras. 
As such, we can give a definition of them in terms of axial algebras.

\begin{definition}
A \emph{Majorana algebra} is a primitive axial algebra $A$ of Monster 
type over $\mathbb{R}$ such that
\begin{enumerate}
\item[M$1$] $A$ has a positive definite Frobenius form $(\cdot,\cdot)$; 
furthermore, $(a,a)=1$ for every axis $a$.
\item[M$2$] \emph{Norton's inequality} holds. That is, for all 
$x,y\in A$,
\[
(x\cdot y,x\cdot y)\leq(x\cdot x,y\cdot y).
\]
\end{enumerate}
\end{definition}

In some papers, the M2 axiom is not assumed and in others additional 
axioms on the subalgebras are assumed such as the $\textup{M}8$ axiom, 
which we will explain later in Section \ref{sec:dihedral}.

\subsection{Gradings and automorphisms}

The key property that axial algebras and Majorana algebras generalise 
from the Griess algebra is that there is a natural link between 
involutory automorphisms and axes.  This link occurs precisely when 
we have a graded fusion law.

\begin{definition}
The fusion law $\mathcal{F}$ is \emph{$T$-graded}, where $T$ is a 
finite abelian group, if $\mathcal{F}$ has a partition 
$\mathcal{F}=\cup_{t\in T}\mathcal{F}_t$ such that
\[
\mathcal{F}_s\star\mathcal{F}_t\subseteq\mathcal{F}_{st}
\]
for all $s,t\in T$.
\end{definition}

Note that, in the same way as we allow trivial eigenspaces, we also 
allow empty parts in the partition in the above definition.

Let $A$ be an algebra and $a\in A$ an $\mathcal{F}$-axis (we do not 
require $A$ to be an axial algebra here). If $\mathcal{F}$ is 
$T$-graded, then this induces a \emph{$T$-grading} on $A$ with 
respect to the axis $a$. The weight $t$ subspace $A_t$ of $A$ is
\[
A_t=A_{\mathcal{F}_t}=\bigoplus_{\lambda\in\mathcal{F}_t}A_\lambda.
\]

This leads to automorphisms of the algebra. Let $T^*$ denote the 
group of linear characters of $T$. That is, the homomorphisms from 
$T$ to $\mathbb{F}^\times$. For $\chi\in T^*$, we define a map 
$\tau_a(\chi)\colon A\to A$ by
\[
v\mapsto\chi(t) v
\]
for $v\in A_t$ and extend linearly to $A$. Since $A$ is $T$-graded, 
this map $\tau_a(\chi)$ is an automorphism of $A$. Furthermore, the 
map sending $\chi$ to $\tau_a(\chi)$ is a homomorphism from $T^*$ 
to $\Aut(A)$. The subgroup $T_a:=\mathrm{Im}(\tau_a)$ of $\Aut(A)$ 
is called the \emph{axis subgroup} corresponding to $a$.

We are particularly interested in $\mathbb{Z}_2$-graded fusion laws. 
In this case, we write $\mathbb{Z}_2$ as $\{+,-\}$ with the usual 
multiplication of signs. For example, the Monster fusion law 
$\mathcal{F}$ is $\mathbb{Z}_2$-graded where 
$\mathcal{F}_+=\{1,0,\frac{1}{4}\}$ and 
$\mathcal{F}_-=\{\frac{1}{32}\}$. 

When the fusion law is $\mathbb{Z}_2$-graded and 
$\mathrm{char}(\mathbb{F})\neq 2$, $T^*=\{\chi_1,\chi_{-1}\}$, 
where $\chi_1$ is the trivial character and $\chi_{-1}$ is the 
alternating character of $T=\mathbb{Z}_2$. Here the axis subgroup 
contains just one non-trivial automorphism, 
$\tau_a:=\tau_a(\chi_{-1})$.  We call this the 
\emph{axial involution}, or \emph{Miyamoto involution}, associated 
to $a$. It is given by the linear extension of
\[
v^{\tau_a}=\begin{cases}
v&\mbox{if }v\in A_+;\\
-v&\mbox{if }v\in A_-.
\end{cases}
\]

Let $Y\subseteq X$ be a set of axes in $A$. We define
\[
G(Y):=\langle T_a:a\in Y\rangle.
\]
We call $G(X)$ the \emph{Miyamoto group}.

For a subset $Y\subseteq X$ of axes, we define $\overline{Y}=Y^{G(Y)}$. 
By \cite[Lemma 3.5]{axialstructure}, $G(\overline{Y})=G(Y)$ and so 
$\overline{Y}^{G(\overline{Y})}=\overline{Y}$. We call $\overline{Y}$ 
the \emph{closure} of $Y$ and we say that $Y$ is \emph{closed} if 
$Y=\overline{Y}$.

\begin{table}[p]
\setlength{\tabcolsep}{4pt}
\renewcommand{\arraystretch}{1.5}
\centering
\footnotesize
\begin{tabular}{c|c|c}
Type & Basis & Products \& form \\ \hline
$2\textrm{A}$ & \begin{tabular}[t]{c} $a_0$, $a_1$, \\ $a_\rho$ 
\end{tabular} & 
\begin{tabular}[t]{c}
$a_0 \cdot a_1 = \frac{1}{8}(a_0 + a_1 - a_\rho)$ \\
$a_0 \cdot a_\rho = \frac{1}{8}(a_0 + a_\rho - a_1)$ \\
$(a_0, a_1) = (a_0, a_\rho)= (a_1, a_\rho) = \frac{1}{8}$
\vspace{4pt}
\end{tabular}
\\
$2\textrm{B}$ & $a_0$, $a_1$ &
\begin{tabular}[t]{c}
$a_0 \cdot a_1 = 0$ \\
$(a_0, a_1) = 0$
\vspace{4pt}
\end{tabular}
\\
$3\textrm{A}$ & \begin{tabular}[t]{c} $a_{-1}$, $a_0$, \\ $a_1$, $u_\rho$ \end{tabular} &
\begin{tabular}[t]{c}
$a_0 \cdot a_1 = \frac{1}{2^5}(2a_0 + 2a_1 + a_{-1}) - \frac{3^3\cdot5}{2^{11}} u_\rho$ \\
$a_0 \cdot u_\rho = \frac{1}{3^2}(2a_0 - a_1 - a_{-1}) + \frac{5}{2^{5}} u_\rho$ \\
$u_\rho \cdot u_\rho = u_\rho$, $(a_0, a_1) = \frac{13}{2^8}$ \\
$(a_0, u_\rho) = \frac{1}{4}$, $(u_\rho, u_\rho) = \frac{2^3}{5}$ 
\vspace{4pt}
\end{tabular}
\\
$3\textrm{C}$ & \begin{tabular}[t]{c} $a_{-1}$, $a_0$, \\ $a_1$ \end{tabular} &
\begin{tabular}[t]{c}
$a_0 \cdot a_1 = \frac{1}{2^6}(a_0 + a_1 - a_{-1})$ \\
$(a_0, a_1) = \frac{1}{2^6}$
\vspace{4pt}
\end{tabular}
\\
$4\textrm{A}$ & \begin{tabular}[t]{c} $a_{-1}$, $a_0$, \\ $a_1$, $a_2$ \\ $v_\rho$ \end{tabular} &
\begin{tabular}[t]{c}
$a_0 \cdot a_1 = \frac{1}{2^6}(3a_0 + 3a_1 - a_{-1} - a_2 - 3v_\rho)$ \\
$a_0 \cdot v_\rho = \frac{1}{2^4}(5a_0 - 2a_1 - a_2 - 2a_{-1} + 3v_\rho)$ \\
$v_\rho \cdot v_\rho = v_\rho$, $a_0 \cdot a_2 = 0$ \\
$(a_0, a_1) = \frac{1}{2^5}$, $(a_0, a_2) = 0$\\
$(a_0, v_\rho) = \frac{3}{2^3}$, $(v_\rho, v_\rho) = 2$
\vspace{4pt}
\end{tabular}
\\
$4\textrm{B}$ & \begin{tabular}[t]{c} $a_{-1}$, $a_0$, \\ $a_1$, $a_2$ \\ $a_{\rho^2}$ \end{tabular} &
\begin{tabular}[t]{c}
$a_0 \cdot a_1 = \frac{1}{2^6}(a_0 + a_1 - a_{-1} - a_2 + a_{\rho^2})$ \\
$a_0 \cdot a_2 = \frac{1}{2^3}(a_0 + a_2 - a_{\rho^2})$ \\
$(a_0, a_1) = \frac{1}{2^6}$, $(a_0, a_2) = (a_0, a_{\rho^2})= \frac{1}{2^3}$
\vspace{4pt}
\end{tabular}
\\
$5\textrm{A}$ & \begin{tabular}[t]{c} $a_{-2}$, $a_{-1}$,\\ $a_0$, $a_1$,\\ $a_2$, $w_\rho$ \end{tabular} &
\begin{tabular}[t]{c}
$a_0 \cdot a_1 = \frac{1}{2^7}(3a_0 + 3a_1 - a_2 - a_{-1} - a_{-2}) + w_\rho$ \\
$a_0 \cdot a_2 = \frac{1}{2^7}(3a_0 + 3a_2 - a_1 - a_{-1} - a_{-2}) - w_\rho$ \\
$a_0 \cdot w_\rho = \frac{7}{2^{12}}(a_1 + a_{-1} - a_2 - a_{-2}) + \frac{7}{2^5}w_\rho$ \\
$w_\rho \cdot w_\rho = \frac{5^2\cdot7}{2^{19}}(a_{-2} + a_{-1} + a_0 + a_1 + a_2)$ \\
$(a_0, a_1) = \frac{3}{2^7}$, $(a_0, w_\rho) = 0$, $(w_\rho, w_\rho) = \frac{5^3\cdot7}{2^{19}}$
\vspace{4pt}
\end{tabular}
\\
$6\textrm{A}$ & \begin{tabular}[t]{c} $a_{-2}$, $a_{-1}$,\\ $a_0$, $a_1$,\\ $a_2$, $a_3$ \\ $a_{\rho^3}$, $u_{\rho^2}$ \end{tabular} &
\begin{tabular}[t]{c}
$a_0 \cdot a_1 = \frac{1}{2^6}(a_0 + a_1 - a_{-2} - a_{-1} - a_2 - a_3 + a_{\rho^3}) + \frac{3^2\cdot5}{2^{11}}u_{\rho^2}$ \\
$a_0 \cdot a_2 = \frac{1}{2^5}(2a_0 + 2a_2 + a_{-2}) - \frac{3^3\cdot5}{2^{11}}u_{\rho^2}$ \\
$a_0 \cdot u_{\rho^2} = \frac{1}{3^2}(2a_0 - a_2 + a_{-2}) + \frac{5}{2^5}u_{\rho^2}$ \\
$a_0 \cdot a_3 = \frac{1}{2^3}(a_0 + a_3 - a_{\rho^3})$, $a_{\rho^3} \cdot u_{\rho^2} = 0$\\
$(a_0, a_1) = \frac{5}{2^8}$, $(a_0, a_2) = \frac{13}{2^8}$ \\
$(a_0, a_3) = \frac{1}{2^3}$, $(a_{\rho^3}, u_{\rho^2}) = 0$,
\end{tabular}
\end{tabular}
\caption{Norton-Sakuma algebras}\label{tab:sakuma}
\end{table}

\subsection{Subalgebras generated by two axes}\label{sec:dihedral}

Since the defining property of axial algebras is that they are generated 
by a set of axes, it is natural to ask: What are the axial algebras that 
are generated by just two axes?  We call such axial algebras 
\emph{$2$-generated} and, if the fusion law is $\mathbb{Z}_2$-graded, we 
also call them \emph{dihedral} because the Miyamoto group in this case 
is dihedral. 

In the Griess algebra, the dihedral subalgebras, called 
\emph{Norton-Sakuma algebras}, were investigated by Norton and shown to 
be one of nine different types \cite{C85}.  In particular, for each pair 
of axes $a_0$, $a_1$ in the Griess algebra, the isomorphism class of the 
subalgebra which they generate is determined by the conjugacy class in 
the Monster of the product $\tau_{a_0} \tau_{a_1}$ of the two 
involutions $\tau_{a_0}$ and $\tau_{a_1}$ associated to the axes. The 
nine different types are: $1\textup{A}$ (when $a_0=a_1$), 
$2\textup{A}$, $2\textup{B}$, $3\textup{A}$, $3\textup{C}$, 
$4\textup{A}$, $4\textup{B}$, $5\textup{A}$ and $6\textup{A}$.

The algebra $1\textup{A}$ is just one dimensional, but the remaining 
eight Norton-Sakuma algebras are given in Table \ref{tab:sakuma} whose 
content we will now explain.  The notation is from 
\cite[Section 2]{seress}.  Let $n\textrm{L}$ be one of the dihedral 
algebras.  Since its generating axes $a_0$ and $a_1$ give involutions 
$\tau_{a_0}$ and $\tau_{a_1}$ in the Monster, we have the dihedral group 
$D_{2n} \cong \langle \tau_{a_0}, \tau_{a_1} \rangle$ acting as 
automorphisms of $n\textrm{L}$ (possibly with a kernel).  In particular, 
let $\rho = \tau_{a_0}\tau_{a_1}$.  We define
\[
a_{\epsilon + 2k} = a_\epsilon^{\rho^k}
\]
for $\epsilon = 0,1$.  It is clear that these $a_i$ are all axes as they 
are conjugates of $a_0$ or $a_1$.  The orbits of $a_0$ and $a_1$ under 
the action of $\rho$ (in fact, under the action of $D_{2n}$) have the 
same size.  If $n$ is even, then these two orbits have size $\frac{n}{2}$ 
and are disjoint, whereas if $n$ is odd, the orbits coincide and have 
size $n$. The map $\tau$ associates an involution to each axis $a$ and 
$\tau_a^g = \tau_{a^g}$ for all $g \in \Aut(n\textrm{L})$.  In almost all 
cases, the axes $a_i$ are not enough to span the algebra.  We index the 
additional basis elements by powers of $\rho$.  Using the action of 
$D_{2n}$, it is enough to just give the products in Table 
\ref{tab:sakuma} to fully describe each algebra.  The axes in each 
algebra are primitive and each algebra admits a Frobenius form that is
 non-zero on the set of axes and invariant under the Miyamoto group; the 
 values for this are also listed in the table.

Amazingly the classification of dihedral algebras also holds, and is 
known as Sakuma's theorem \cite{sakuma}, if we replace the Griess algebra 
by the weight two subspace $V_2$ of a vertex operator algebra (VOA) 
$V = \bigoplus_{n = 0}^\infty V_n$ over $\mathbb{R}$ where 
$V_0 = \mathbb{R} 1$ and $V_1 = 0$ (those of OZ-type).  After Majorana 
algebras were defined generalising such VOAs, the result was reproved for 
Majorana algebras by Ivanov, Pasechnik, Seress and Shpectorov in 
\cite{IPSS}.  In the paper introducing axial algebras, the result was 
also shown to hold in axial algebras of Monster type over a field of 
characteristic $0$ which have a Frobenius form \cite{Axial1}.  It is 
conjectured that the Frobenius form is not required.

\begin{conjecture}\label{conj:dihedral}
A dihedral axial algebra of Monster type over a field of characteristic 
$0$ is one of the nine Norton-Sakuma algebras.\footnote{A proof of this 
conjecture was recently announced by Franchi, Mainardis and Shpectorov 
at the Axial Algebra Focused Workshop in Bristol in May 2018.}
\end{conjecture}

For Majorana algebras, the following axiom is also often assumed.

\begin{enumerate}
\item[$\mathrm{M}8$] Let $a_i \in X$ be axes for $0 \leq i \leq 2$.  If 
$a_0$ and $a_1$ generate a dihedral subalgebra of type $2\mathrm{A}$, 
then $a_\rho \in X$ and $\tau_{a_\rho} = \tau_{a_0}\tau_{a_1}$.  
Conversely, if $\tau_{a_0}\tau_{a_1}\tau_{a_2}=1$, then $a_0$ and $a_1$ 
generate a dihedral subalgebra of type $2\mathrm{A}$ and $a_2 = a_\rho$.
\end{enumerate}

This severely restricts the possible configuration of subalgebras. We 
will explain this later in Section \ref{sec:shape} once we have 
introduced shapes.

Seress \cite{seress} also assumed that the map $\tau$ was a bijection 
between the set of axes $X$ and a union of conjugacy classes of 
involutions in $G$.  Moreover the following conditions which restrict 
the intersections of subalgebras were also assumed.  Let 
$a_i, b_i \in X$ and $\rho(a_0, a_1) = \tau_{a_0} \tau_{a_1}$.

\begin{enumerate}
\item[$2\mathrm{Aa}$] If $\tau_{a_0} \tau_{a_1} \tau_{a_2} = 1$ and 
$\langle a_0, a_1 \rangle \cong 2\A$, then 
$a_2 \in \langle a_0, a_1 \rangle$ and $a_2 = a_{\rho}$.

\item[$2\mathrm{Ab}$] If $\langle a_0, a_1 \rangle$ and 
$\langle b_0, b_1 \rangle$ are both of type $2\A$ and 
$\langle \rho(a_0, a_1) \rangle = \langle \rho(b_0, b_1) \rangle$, 
then the extra basis elements $a_\rho(a_0, a_1)$ and 
$a_\rho(b_0, b_1)$ are equal.

\item[$3\A$] If $\langle a_0, a_1 \rangle$ and 
$\langle b_0, b_1 \rangle$ are both of type $3\A$ and 
$\langle \rho(a_0, a_1) \rangle = \langle \rho(b_0, b_1) \rangle$, 
then the extra basis elements $u_\rho(a_0, a_1)$ and 
$u_\rho(b_0, b_1)$ are equal.

\item[$4\A$] If $\langle a_0, a_1 \rangle$ and 
$\langle b_0, b_1 \rangle$ are both of type $4\A$ and 
$\langle \rho(a_0, a_1) \rangle = \langle \rho(b_0, b_1) \rangle$, 
then the extra basis elements $v_\rho(a_0, a_1)$ and 
$v_\rho(b_0, b_1)$ are equal.

\item[$5\A$] If $\langle a_0, a_1 \rangle$ and 
$\langle b_0, b_1 \rangle$ are both of type $5\A$ and 
$\langle \rho(a_0, a_1) \rangle = \langle \rho(b_0, b_1) \rangle$, 
then the extra basis elements $w_\rho(a_0, a_1)$ and 
$w_\rho(b_0, b_1)$ are equal up to a change of sign.
\end{enumerate}

We can also consider a wider class of axial algebras.  Axial 
algebras of Jordan type $\eta$ were considered in \cite{Axial2}.  
Here there are just three eigenvalues, $1$, $0$ and $\eta$.  When 
$\eta \neq \frac{1}{2}$, all algebras were classified and they 
relate to $3$-transposition groups.  The \emph{Ising fusion law} 
$\Phi(\alpha, \beta)$ is given in Table \ref{tab:Ising}.
\begin{table}[!htb]
\setlength{\tabcolsep}{4pt}
\renewcommand{\arraystretch}{1.5}
\centering
\begin{tabular}{c||c|c|c|c}
 & $1$ & $0$ & $\alpha$ & $\beta$ \\ \hline \hline
$1$ & $1$ &  & $\alpha$ & $\beta$ \\ \hline
$0$ &  & $0$ &$\alpha$ & $\beta$ \\ \hline
$\alpha$ & $\alpha$ & $\alpha$ & $1, 0$ & $\beta$ \\ \hline
$\beta$ & $\beta$  & $\beta$ & $\beta$ & $1, 0, \alpha$ 
\end{tabular}
\caption{Ising fusion law $\Phi(\alpha, \beta)$}\label{tab:Ising}
\end{table}

In particular, note that the Monster fusion law is just 
$\Phi(\frac{1}{4}, \frac{1}{32})$.  In \cite{felix}, Rehren studies 
dihedral axial algebras over $\Phi(\alpha, \beta)$ with a Frobenius 
form and shows that the nine algebras above can be generalised and 
live in families which exist for values of $\alpha$ and $\beta$ 
lying in certain varieties.  It turns out that 
$(\alpha, \beta) = (\frac{1}{4}, \frac{1}{32})$ is a distinguished 
point.


\section{Shapes}\label{sec:shape}

The shape of an axial algebra $A$ specifies which $2$-generated 
subalgebras arise in $A$. Clearly, a precondition for such a 
description is the knowledge of the possible $2$-generated 
algebras; that is, for the class of axial algebras under 
consideration we either should have classified all $2$-generated 
algebras or, minimally, we should have an explicit list of such
algebras that we want to allow in $A$. 

Note that the 2-generated algebras should be classified not up to 
an abstract algebra isomorphism, but rather up to the (unique 
possible) isomorphism sending the two generating axes of one 
algebra to the two generating axes of the other algebra. That is, 
we consider the $2$-generated algebras as having marked generators 
and isomorphisms must respect them: if $B$ has marked generators 
$a$ and $b$ and $B'$ has marked generators $a'$ and $b'$ then 
$(B,(a,b))$ is isomorphic to $(B',(a',b'))$ only if there is an 
isomorphism $\phi \colon B\to B'$ such that $\phi(a)=a'$ and $\phi(b)=b'$.
In principle, an algebra may have non-equivalent pairs of 
generators and then this algebra must accordingly appear on the 
list several times. Note that for algebras of Monster type, Sakuma 
theorem classifies dihedral algebras exactly in this sense: in each 
of the eight Norton-Sakuma algebra the marked generators are $a_0$ 
and $a_1$ and any other pairs of generators is equivalent to 
$(a_0,a_1)$.  Therefore, in order to motivate the general case, we consider first the case of an axial algebra of Monster type.


Let $A$ be an axial algebra of Monster type and suppose that $X$ is a set of axes which generates $A$.  Note that by enlarging our set $X$, we may assume that $X$ is closed under the action of the Miyamoto group $G$ of $A$.

\begin{lemma}\label{Gfaithful}
The action of $G$ on $X$ is faithful.
\end{lemma}
\begin{proof}
Suppose that $g \in G$ fixes all the axes in $X$.  However, the subspace of $A$ fixed by $g$ is a subalgebra and, since it contains $X$, it contains the whole algebra $A$.
\end{proof}

As $G$ is a group of automorphisms of $A$, if $a, b \in X$ generate a dihedral subalgebra $B$, then, for any $g \in G$, the subalgebra generated by $a^g, b^g$ is isomorphic to $B$.  In this way, we obtain the \emph{shape} of the algebra which is a map $S$ from the set of $G$-orbits on $X \times X$ to the set of dihedral algebras.

Given a pair of axes $(a, b)$, let $D_{a,b}$ be the dihedral group generated by $\tau_a$ and $\tau_b$.  Define $X_{a,b} = a^D \cup b^D$, where $D := D_{a,b}$.  It is clear that $D_{a,b} = D_{b,a}$ and $X_{a,b} = X_{b,a}$.

A Sakuma algebra has type $n\textrm{L}$.  We wish to show that $n$ can be determined solely from the action of the dihedral group $D_{a,b}$.

\begin{lemma}\label{dihedralorbs}
Let $a,b \in X$ and $D := D_{a,b}$.  Then, $|a^D| = |b^D|$.  If $a$ and $b$ are in the same orbit, then the length of this orbit is $1$, $3$, or $5$.  Otherwise, if $a$ and $b$ are in different orbits, then the length of each orbit is $1$, $2$, or $3$.  Moreover, the Sakuma algebra generated by $a$ and $b$ has type $n\textrm{L}$, where $n = |X_{a,b}|$.
\end{lemma}
\begin{proof}
A direct proof would be long and computational.  So instead we observe that each Norton-Sakuma algebra is contained in the Griess algebra and there we have a bijection between axes and $2\A$-involutions in the Monster $M$.  So, we may take the dihedral subgroup $H \leq M$ generated by the involutions associated to each axis (in the Griess algebra).  In particular, up to the kernel, the action of $H$ on $X$ is the same as the action of $D$ on $X$.

Since in the Griess algebra we have a bijection between axes and $2\A$-involutions and $\tau_x^g = \tau_{x^g}$ for $g \in H$, we may consider the orbits of involutions in $H$ rather than the orbits of axes.  The result now follows from properties of dihedral groups and Sakuma's theorem.
\end{proof}

Thus, when we know the action of $G$ on $X$, $n$ is known for each orbit and the shape is determined by choices of $\textrm{L}$.  Furthermore, these choices are not independent.


If $a,b,c,d \in X$ then we say $(a,b)$ \emph{dominates} $(c,d)$ if $c,d \in X_{a,b}$.  In particular, when this happens, $X_{c,d} \subseteq X_{a,b}$ and $D_{c,d} \leq D_{a,b}$.  Note also that the subalgebra $\langle c,d \rangle$ is contained in $\langle a, b\rangle$.  Hence, if $(a,b)$ dominates $(c,d)$, then the choice of dihedral subalgebra $\langle a,b \rangle$ determines the choice for $\langle c,d \rangle$.  For the Monster fusion law, we have the following non-trivial inclusions
\[
\begin{array}{c|c}
\langle a,b \rangle & \langle c,d \rangle \\
\hline
4\textrm{A} & 2\textrm{B} \\
4\textrm{B} & 2\textrm{A} \\
6\textrm{A} & 2\textrm{A} \\
6\textrm{A} & 3\textrm{A}
\end{array}
\]
Note that here, not only does the choice of $\langle a,b \rangle$ determine the choice for $\langle c,d \rangle$, but also the choice for $\langle c,d\rangle$ uniquely determines the choice for $\langle a,b\rangle$.  Additionally, note that the pair $(a,b)$ always dominates $(b,a)$ and vice versa, so in the next concept which describes the totality of choices, we may just work with the set $\{a,b\}$ instead of the pairs $(a,b)$ and $(b,a)$.  Notice also that since $X_{a,b} = X_{b,a}$, the concept of domination is not affected by the switch to sets.

Let ${X\choose 2}$ denote the set of $2$-subsets of $X$.  The orbits of $G$ on ${X\choose 2}$ are the vertices of a directed graph, called the \emph{shape graph}, with the edges given by domination.  By the above comment, there is at most one choice of dihedral subalgebra for each weakly connected component (i.e. a connected component of the underlying undirected graph).  So, the shape of an algebra is fully described by assigning one dihedral subalgebra per weakly connected component.  Sometimes there is no choice for a given component.  Namely, when that component contains a $6\A$, or $5\A$.

Additionally, if the M8 axiom is assumed, then this further restricts the allowable shapes.  Suppose that $a$ and $b$ are such that $X_{a,b} = \{a,b\}$ and $\tau_a$ and $\tau_b$ are the involutions associated to $a$ and $b$.  Then $\tau_a\tau_b$ has order two.  If $\tau_a\tau_b$ is in the image of the $\tau$-map, then M8 demands that the dihedral subalgebra $B = \langle a,b \rangle$ generated by $a$ and $b$ be a $2\A$.  Conversely, if $\tau_a\tau_b$ is not in the image $\tau$, then the dihedral subalgebra $B$ must be a $2\B$.  In both cases, this defines the shape on the connected component containing the orbit of $\{a,b\}$.  However, the only connected components which don't contain any dihedral subalgebras with $n=2$ are those which just contain a single dihedral subalgebra with $n=3$.  So, if the M8 condition is assumed the only choice over a shape is choosing whether those connected components which consist of a single $3\L$ are $3\A$, or $3\C$.

We now turn to the general case of a fusion law $\mathcal{F}$ which is $T$-graded and an abstract group of permutations $G$ acting faithfully on a set $X$.  We are thinking of an unknown axial algebra $A$ with fusion law $\mathcal{F}$ and the action of the Miyamoto group on the axes being the action of (a normal subgroup of) $G$ on $X$.  It is clear that we may just consider actions up to isomorphism.  We will define analogous concepts to above.

\begin{definition}
A map $\tau\colon  X \times T^* \to G$ is called a \emph{$\tau$-map} if for all $x \in X$, $\chi \in T^*$, $g \in G$
\begin{enumerate}
\item $\tau_x \colon  T^* \to G$ is a group homomorphism;
\item $\tau_x(\chi)^g = \tau_{x^g}(\chi)$.
\end{enumerate}
We call the image $G_0:=\langle \tau_x(\chi) : x \in X, \chi \in T^* \rangle\unlhd G$ the \emph{Miyamoto group} of $\tau$.
\end{definition}

As previously, we define $T_x := \langle \tau_x(\chi) : \chi \in T^* \rangle\leq G_0$.

\begin{lemma}\label{taufix}
$T_x \subseteq Z(G_x)$.
\end{lemma}
\begin{proof}
Let $g \in G_x$.  Then for $\chi \in T^*$,
\[
[\tau_x(\chi), g] = \tau_x(\chi)^{-1} \tau_x(\chi)^g = \tau_x(\chi)^{-1} \tau_{x^g}(\chi) =\tau_x(\chi)^{-1} \tau_{x}(\chi) = 1\qedhere
\]
\end{proof}

We define $D = D_{a,b} := \langle T_a, T_b \rangle$ for $a,b \in X$. Unlike the Monster type case, $D$ does not have to be a dihedral group.  In an $\mathcal{F}$-axial algebra, $D_{a,b}$ acts on the subalgebra $\langle a,b \rangle$.  Suppose that we know a list $\mathcal{L}$ of $2$-generated subalgebras with marked generators for the fusion law $\mathcal{F}$.  We wish to impose conditions on $\tau$ so that $D_{a,b}$ has an action on $X_{a,b}:=a^D \cup b^D$ which is an action observed on the axes of some $2$-generated algebra in our list.  Otherwise, $\tau$ cannot lead to a valid $\mathcal{F}$-axial algebra.

\begin{definition}
A $\tau$-map $\tau \colon X \times T^* \to G$ is called \emph{admissible} if for every set $\{a,b\} \in {X \choose 2}$, the action of $D_{a,b}$ on $X_{a,b}$ agrees with at least one algebra in the list $\mathcal{L}$.
\end{definition}

For example, let $\mathcal{F}$ be the Monster fusion law.  Then a complete list of the dihedral subalgebras is known.  In particular, the orbits of $a$ and $b$ under $D$ must have the properties given in Lemma \ref{dihedralorbs}.  That is,
\begin{enumerate}
\item $k := |a^D| = |b^D|$.
\item If $a$ and $b$ are in the same $D$-orbit, then $k = 1$, $3$, or $5$.
\item If $a$ and $b$ are in different $D$-orbits, then $k = 1$, $2$, or $3$.
\end{enumerate}

From now on, we only consider admissible $\tau$-maps.  The normaliser $N = N_{\mathrm{Sym}(X)}(G)$ of the action of $G$ on $X$ acts on the set of admissible $\tau$-maps by
\[
\tau \mapsto \tau^n \qquad \mbox{where } (\tau^n)_x(\chi) := \tau_{x^{n^{-1}}}(\chi)^n
\]
for $n \in N$.  Note that, by the definition of a $\tau$-map, $G$ acts trivially on each $\tau$.  So an action of $N/G$ is induced on the set of $\tau$-maps.  Thus, we may just consider admissible $\tau$-maps up to the action of $N/G$.

Next we introduce domination.

\begin{definition}
For $\{a,b\},\{c,d\} \in {X \choose 2}$, we say $\{a,b\}$ \emph{dominates} $\{c,d\}$ if $c,d \in X_{a,b}$.
\end{definition}

\begin{definition}
The \emph{shape graph} $\Gamma$ is a directed graph with vertices given by orbits of $G$ on ${X \choose 2}$ and edges given by domination between pairs from those orbits.
\end{definition}

%
%

As observed above, for the Monster fusion law, any one choice of $2$-generated subalgebra for a weakly connected component of the shape graph determines all other $2$-generated algebras in that component.  For a general fusion law, the dominated algebra may not always determine the larger algebra uniquely.  However, the larger, dominating algebra always determines the smaller algebra. We will call these the \emph{domination restrictions}.

\begin{definition}
Given an abstract group $G$ acting faithfully on a set $X$ and an admissible $\tau$-map, a \emph{shape} on $X$ is a set of choices of $2$-generated algebra for all orbits of $G$ on ${X \choose 2}$ which satisfy the domination restrictions.
\end{definition}

Given a group $G$ acting faithfully on a set $X$ and an admissible $\tau$-map $\tau$, we may consider all the possible shapes.  Let $K = \stab_N(\tau)$.  As noted above, $G$ acts trivially on each $\tau$, and in fact it also fixes every shape. On the other hand, $K$ (or rather $K/G$) permutes the $G$-orbits of ${X \choose 2}$, and so 
may act non-trivially on the set of shapes.  So, we may consider shapes for $\tau$ up to the action of $K$.

In summary, given an action of a group $G$ on a putative set of axes $X$, we can determine all the possible admissible $\tau$-maps. Given a particular $\tau$-map $\tau$, we can further determine all the possible shapes that an axial algebra with Miyamoto group $G_0$ and $\tau$-map $\tau$ could have.

\section{Useful lemmas}\label{sec:preliminaries}

In this section, we will discuss some properties which must hold in axial algebras.  We will use these later in the algorithm to discover relations and to build up eigenspaces.

Recall that we adopt the notation that for a subset $I \subseteq \mathcal{F}$, 
\[
A_I = \bigoplus_{\lambda \in I} A_\lambda
\]
We begin by noting that, since we allow $I$ to be a subset, we can add and intersect the $A_I$.

\begin{lemma}\label{sumup&int}
Let $I,J \subseteq \mathcal{F}$, then
\begin{enumerate}
\item[$1.$] $A_I + A_J = A_{I \cup J}$
\item[$2.$] $A_I \cap A_J = A_{I \cap J}$\qed
\end{enumerate}
\end{lemma}

By an abuse of terminology, we will call the $A_I$ eigenspaces of $a$.

\begin{lemma}\label{multiplydown}
Let $a$ be an axis, $I \subseteq \mathcal{F}$, $\lambda \in I$ and $A_I = A_I^a$.  Then, for all $u \in A_I$
\[
ua - \lambda u \in A_{I - \lambda}
\]
\end{lemma}
\begin{proof}
We may decompose $u \in A_I$ as $u = \sum_{\mu \in I} u_\mu$, where $u_\mu \in A_\mu$.  Multiplying by $a$ and subtracting $\lambda u$, we have
\begin{align*}
ua - \lambda u &= \sum_{\mu \in I} u_\mu a - \lambda u \\
&= \sum_{\mu \in I} (\mu - \lambda) u_\mu
\end{align*}
Since the coefficient of $u_\lambda$ is zero, the above is in $A_{I - \lambda}$.
\end{proof}

Recall that we extended the operation $\star$ to all subsets of $\mathcal{F}$, turning the fusion law into a magma.  Moreover, the eigenspaces $A_I$ satisfy the fusion law.  However, not all fusion rules on subsets are equally useful for our algorithm. In particular, assuming that $\mathcal{F}$ is $T$-graded, we only need to consider $I$ fully 
contained in a part $\mathcal{F}_t$ for some $t\in T$. We call such subsets \emph{pure}. 

\begin{definition}
Let $I \subseteq \mathcal{F}_s$ and $J \subseteq \mathcal{F}_t$ for $s, t \in T$.  We define a fusion rule $I\star J = K$ to be \emph{useful} if 
\begin{enumerate}
\item $K \subsetneqq \mathcal{F}_{s \star t}$; and 
\item there does not exist $I \subsetneqq I'\subseteq \mathcal{F}_s$, or $J \subsetneqq J' \subseteq \mathcal{F}_t$ such that
\[
I' \star J = K \qquad \mbox{or} \qquad I \star J' = K
\]
\end{enumerate}
\end{definition}

In particular, given a useful fusion rule $I \star J = K$, if we require it to hold, all other rules $X \star Y = K$ for subsets $X \subseteq I$ and $Y \subseteq J$ will automatically be satisfied.  In this way, it is enough to impose just the useful fusion rules and the grading to capture all the information from the fusion law. 

To calculate the useful fusion rules for any fusion law $\mathcal{F}$ we begin by writing out the expanded fusion table for all pure subsets of $\mathcal{F}$ with rows and columns partially ordered by inclusion.  We then consider all sets $K$ which occur as entries in the table.  The useful rules are precisely those where $K$ is not a full part $\mathcal{F}_t$, for $t\in T$, and it does not appear in the expanded table below in that column, or to the right in that row.  Doing this to the Monster fusion law results in the following list.

\begin{lemma}\label{Monsteruseful}
The useful fusion rules for the Monster fusion table are
\[
\begin{gathered}
1 \star 0 = \emptyset \qquad 1 \star \{1,0\} = 1 \qquad 1 \star \{0,\tfrac{1}{4}\} = \tfrac{1}{4} \qquad 1 \star \{1,0,\tfrac{1}{4}\} = \{1,\tfrac{1}{4}\} \\
0 \star \{1,0\} = 0 \qquad 0 \star \{1,\tfrac{1}{4}\} = \tfrac{1}{4} \qquad 0 \star \{1,0,\tfrac{1}{4}\} = \{0,\tfrac{1}{4}\} \\
\tfrac{1}{4} \star \tfrac{1}{4} = \{1,0\} \qquad \tfrac{1}{4} \star \{1,0\} = \tfrac{1}{4} \\
\{1,0\} \star \{1,0\} = \{1,0\} \qquad \{1,0\} \star \{1,\tfrac{1}{4}\} = \{1,\tfrac{1}{4}\} \qquad \{1,0\} \star \{0,\tfrac{1}{4}\} = \{0,\tfrac{1}{4}\}
\end{gathered}
\]
\end{lemma}

Note that all useful fusion rules for the Monster fusion law come from the even part.  That is because the values of $\star$ involving the odd part $\{\tfrac{1}{32}\}$ are fully determined by the grading.

If $A$ is primitive, then for an axis $a$, $G_a$ certainly fixes every vector in $A_1^a$.  We now describe another trick which uses this weaker condition.

\begin{lemma}\label{wh-w}
Let $1 \in I \subset \mathcal{F}$ and $u \in A_I(a)$ for an axis $a$.  Suppose further that $G_a$ fixes every vector in $A_1^a$.  Then, for all $g$ in the stabiliser $G_a$,
\[
u^g - u \in A_{I - 1}
\]
\end{lemma}
\begin{proof}
We decompose $u = \sum_{\mu \in I} u_\mu$ with respect to the eigenspaces of $a$. Since $g$ fixes $a$, it preserves every eigenspace of $a$.  Furthermore, since $g$ fixes every vector in $A_1^a$, we have the following
\begin{align*}
u^g - u &= \sum_{\mu \in I} u_\mu^g - \sum_{\mu \in I} u_\mu\\
&= u_1^g - u_1 + \sum_{\mu \in I-1} u_\mu^g - u_\mu \\
&= \sum_{\mu \in I-1}  u_\mu^g - u_\mu \in A_{I - 1}\qedhere
\end{align*}
\end{proof}

\section{Algorithm}\label{sec:algorithm}

In this section, we describe our main result which is an algorithm for constructing an axial algebra. A very similar algorithm can also be used to build a module for a known axial algebra. However, we don't want to complicate this paper with extra definitions and so we just deal with the task of constructing an axial algebra. 

In principle, there is no reason to believe that an axial algebra which is generated by a finite set of axes is even finite dimensional.  Clearly, if it is infinite dimensional, our algorithm will not finish.  However, in practice, we can compute a large number of examples as we shall see in Section \ref{sec:results}.

As described in Section \ref{sec:shape}, associated with a $T$-graded $\mathcal{F}$-axial algebra $A$ we have a group $G$ acting faithfully on a set $X$, an admissible $\tau$-map $\tau \colon X \times T^*\to G_0 \unlhd G$ and a shape.  Given such a $G$, $X$, $\tau$ and shape, the algorithm builds an axial algebra $A$ with axes $X$ and Miyamoto group $G_0$.  It does so by defining a partial algebra and completing it step by step into a full algebra.

As input to our algorithm, we take a field $\mathbb{F}$, a $T$-graded fusion law $\mathcal{F}$, a group $G$ acting faithfully on a set $X$, an admissible $\tau$-map $\tau$ and a shape.  These are fixed throughout the rest of this section.

\subsection{Partial algebras}

At the core of the algorithm is a concept which we call a partial algebra.  We write $S^2(V)$ for the symmetric square of $V$.

\begin{definition}
Given a group $G$, a \emph{partial $G$-algebra} is a triple $W = (W, V, \mu)$ where $W$ is a $G$-module over $\mathbb{F}$, $V \subseteq W$ is a $G$-submodule and $\mu \colon S^2(V) \to W$ is a linear map which is $G$-equivariant.
\end{definition}

Where it is clear, we will abuse notation and write $uv$ for $\mu(u,v)$.

\begin{lemma}
Given a $G$-invariant set $Y$ in $W$, there exists a unique smallest submodule $W(Y)$ of $W$ such that
\[
W(Y) = \langle Y \rangle + \mu(S^2(W(Y) \cap V))
\]
\end{lemma}
\begin{proof}
Define $U_0 := \langle Y \rangle$ and inductively define
\[
U_{i+1} := U_i + \mu(S^2(U_i \cap V)).
\]
Then the union of the $U_i$ is $W(Y)$.
\end{proof}

We call $W(Y) = (W(Y), W(Y) \cap V, \mu|_{S^2(W(Y) \cap V)})$ the \emph{partial subalgebra generated by $Y$}.  If $W(Y) = W$, then we say $Y$ \emph{generates} $W$.  For example, an axial algebra $A$ is a partial $G$-algebra, where $G$ is the Miyamoto group, and the set of axes $X$ generates $A$.

\begin{definition}
Let $(W, V, \mu)$ be a partial $G$-algebra and $(W', V', \mu')$ be a partial $G'$-algebra.  A homomorphism of partial algebras is a pair $(\phi, \psi)$ where
\begin{enumerate}
\item $\phi \colon W \to W'$ is a vector space homomorphism such that $\phi(V) \subseteq V'$.
\item $\psi \colon G \to G'$ is a group homomorphism such that 
\[
\phi(w^g) = \phi(w)^{\psi(g)}
\]
for all $w \in W$, $g \in G$.
\item $\phi(\mu(u,v)) = \mu'(\phi(u), \phi(v))$ for all $u,v \in V$.
\end{enumerate}
\end{definition}

In other words, we have the following commutative diagram and additionally the action of $G$ (sometimes acting through $\psi$) commutes with the diagram.

\begin{center}
\begin{tikzcd}[step=large]
S^2(V) \arrow[r, "\mu"] \arrow[d, "\phi \otimes \phi"] & W \arrow[d, "\phi"] & V \arrow[l, hookrightarrow]\arrow[d, "\phi"] \\
S^2(V') \arrow[r, "\mu'"] & W'  & V' \arrow[l, hookrightarrow]
\end{tikzcd}
\end{center}

\subsection{Gluings}

In order to correctly build an axial algebra, we must impose the conditions coming from the shape.  We do this by gluing in subalgebras corresponding to the shape.

First, consider an axial algebra $A$ and let $B$ be a subalgebra in the shape.  Then there is a $K$-submodule $U$ of $A$ such that $\phi \colon U \to B$ is an algebra isomorphism which is invariant under the action of the induced Miyamoto group
\[
K := \langle T_y : y \in Y \rangle
\]
where $Y = X \cap U$ is the subset of axes in $X$ which are in $U$.  However, since $Y$ is a subset of $X$, $K$ does not necessarily act faithfully on $Y$.  Let $N$ be the kernel of the action and $H := K/N$.  Then, the Miyamoto group of $B$ is isomorphic to $H$ and so there exists a group homomorphism $\psi \colon K \to H$ with the property that
\[
\phi(u^g) = \phi(u)^{\psi(g)}
\]
for all $g \in K$, $u \in U$.  With this in mind, we make the following definition.

\begin{definition}
Let $(W, V, \mu)$ be a $G$-partial algebra generated by $X$ and $(W', V', \mu')$ be a partial $H$-algebra generated by a set $X'$.  A \emph{gluing} of $W'$ onto a closed set of axes $Y \subseteq X$ is a homomorphism of partial algebras $(\phi, \psi)$ from the restricted $K$-partial subalgebra $(W(Y), W(Y) \cap V, \mu|_{S^2(W(Y) \cap V)})$ to $(W', V', \mu')$ such that
\begin{enumerate}
\item $K := \langle T_y : y \in Y \rangle \leq G$.
\item $\phi \colon W(Y) \to W'$ is surjective and $\phi(Y) = X'$.
\item $\psi \colon K \to H$ is surjective.
\end{enumerate}
\end{definition}

\subsection{The algorithm}

Our task is to build an algebra of the correct shape.  We will do this by defining a sequence of partial algebras and at each stage `discovering' more of the multiplication.  Throughout our algorithm $W = (W, V, \mu)$ will be a partial $G$-algebra generated by the set $X$, our putative set of axes on which $G$ acts faithfully.  Our algorithm will terminate when $V = W$.  That is, when we know all the multiplication.  We begin with $W$ having basis indexed by the set $X$.  That is, $W$ is a permutation module for the action of $G$ on $X$.  No products are known at this stage, so $V = 0$.  

Throughout our algorithm, we keep track of various (sums of) eigenspaces for each axis.  These are key to finding enough relations to allow our algorithm to terminate.  Recall that the sum of eigenspaces is denoted by $W_I = \bigoplus_{\lambda \in I} W_\lambda$, for a subset $I \subseteq \mathcal{F}$. Note that at any given stage in our algorithm, we may not know the full $\lambda$-eigenspace and so we do not necessarily know the decomposition $W = \bigoplus_{\lambda \in \mathcal{F}} W_\lambda$.  Indeed, we may know that a vector lies in $W_I$, for some $I \subset \mathcal{F}$, but not know how to decompose it into the sum of eigenvectors for eigenvalues $\lambda \in I$.  For this reason, we keep track of sums of eigenspaces $W_I$.  Note that relations are vectors in $W_\emptyset$.  Since $G$ acts on $W$, we may just consider axes and their associated decompositions up to the action of $G$.

It turns out that it is enough to keep track of just the $W_I$, for pure subsets $I$.  That is, the $W_I$ for $I \subseteq \mathcal{F}_t$, for $t \in T$.  We show that this holds, provided we make a mild assumption on the grading group $T$.

Indeed, by assumption, for each axis $a \in X$, there is a decomposition $W = \bigoplus_{t \in T} W_t$.  We claim that we can recover the decomposition $W = \bigoplus_{t \in T/R} W_t$, where $R := \bigcap_{\chi \in T^*} \ker(\chi)$, from the action on $T_a$ on $W$.  Indeed, recall from the definition that $\tau_a(\chi) \in T_a$ acts on $W_t$ by scalar multiplication by $\chi(t)$.  Since this must hold in any axial algebra we build, we can distinguish the $T$-grading up to the kernel $R = \bigcap_{\chi \in T^*} \ker(\chi)$.  If $T^* \cong T$, then $R =1$.  However if $R$ is non-trivial, for example when the characteristic divides $|T|$, or when the field doesn't contain the suitable roots of unity, we can only detect a coarser grading by $T/R \cong T^*$.  Since we may always consider a more coarse grading, from now on, we may assume that $T^* \cong T$ and hence $T_a$ detects the $T$-grading.  Note that for a $\mathbb{Z}_2$-grading, provided the field is not of characteristic two, $-1$ is always in the field and hence we can detect a $\mathbb{Z}_2$-grading using the axis subgroup.

Let $J \subset \mathcal{F}$.  Since we know the decomposition $W = \bigoplus_{t \in T} W_t$, this induces a decomposition $W_J = \bigoplus_{t \in T} W_{J_t}$, where $J_t := J \cap \mathcal{F}_t$.  Now, the only results we will use in our algorithm are those found in Section \ref{sec:preliminaries}, namely, summation and intersection of subspaces, being an eigenvector, obeying the fusion law and, optionally, Lemma \ref{wh-w}.  It is easy to see that for all of these, the information gained about $W_J$ is precisely the sum of the information gained about the $W_{J_t}$.  For example, if $\lambda \in J$, then by Lemma \ref{multiplydown}, $ua-\lambda u \in A_{J-\lambda}$.  But, since we may decompose $u = \sum_{t \in T} u_t$, we have
\[
u_t a - \lambda u_t \in A_{J_t-\lambda} = \begin{cases} A_{J_t} & \mbox{if } \lambda \notin J_t \\
A_{J_t-\lambda}& \mbox{if } \lambda \in J_t \end{cases}
\]
In particular, we recover the only non-trivial result by just considering the pure subset $J_t \subseteq \mathcal{F}_t$.  This justifies our claim that it is enough to keep track of the $W_I$, for pure subsets $I$.

The information for the multiplication, and so also for the eigenspaces, will come from gluing in subalgebras to our partial algebra according to the shape.  In order to fully describe our axial algebra, we must glue in enough subalgebras to cover all those $2$-generated subalgebras given in the shape.  However, we may glue in known subalgebras of the correct shape which are generated by three or more axes.  These have the advantage of containing more information.  (We may also glue in some partial subalgebras, so long as we also glue in enough known subalgebras to cover those given in the shape.)

Since no multiplication is known when we start and $W$ is spanned by the axes, for each gluing of a subalgebra $B$ onto a closed subset of axes $Y$, we have $W(Y) = \langle Y \rangle$ and $\phi$ is the corresponding injection on these axes compatible with the action.

The algorithm has three main stages:
\begin{enumerate}
\item Expansion by adding the formal products of vectors we do not already know how to multiply.
\item Work to discover relations and construct the eigenspaces for the axes.
\item Reduction by factoring out by known relations.
\end{enumerate}

We continue applying these three stages until $V = W$ and our algorithm terminates.  Again, we note that since we use the action of the group, we need only consider subalgebras and axes up to the action of $G$.

If our algorithm does terminate, then we have the following result, which we will prove after describing our algorithm.

\begin{theorem}\label{algorithmthm}
Suppose that the algorithm terminates and returns $A$.  Then $A$ is a \textup{(}not necessarily primitive\textup{)} axial algebra generated by axes $X$ with Miyamoto group $G_0$, $\tau$-map $\tau$ and of the given shape.

Moreover, provided we do not use the optional Lemma $\ref{wh-w}$ in stage $2$ of the algorithm, the algebra is universal.  That is, given any other axial algebra $B$ with the same axes $X$, Miyamoto group $G_0$, $\tau$-map $\tau$ and shape, $B$ is a quotient of $A$.
\end{theorem}

Note that, if we do use Lemma \ref{wh-w} in stage 2 of the above algorithm, then we have assumed that $G_a$ fixes every vector in $A_1^a$ for each axis $a$.  This holds in primitive axial algebras, but not necessarily in the non-primitive case.

\subsubsection*{Stage 1: Expansion}
We expand $W$ to a larger partial algebra $W_\new$ by adding vectors which are the formal products of elements we do not yet know how to multiply.

\begin{description}[leftmargin =0.5cm]
\item[Step 1.] We begin by finding a complement subspace $C$ for $V$ in $W$.  Hence, as a vector space
\[
W = V \oplus C
\]
\end{description}

Wherever possible, we choose $C$ to be a $G$-submodule.  For example, in characteristic $0$, this is always possible.

Since we know the multiplication on $V$ and our multiplication is commutative, we just need to add the products of $V$ with $C$ and products of $C$ with $C$.

\begin{description}[leftmargin =0.5cm]
\item[Step 2.] Form a new partial algebra $W_\new = (W_\new, V_\new, \mu_\new)$ with
\begin{align*}
W_\new &= W \oplus V\otimes C \oplus S^2(C) \\
V_\new &= W
\end{align*}
and $\mu$ extended in the obvious way to $\mu_\new$.
\end{description}

Note that if $C$ is a $G$-submodule, then the summands in $W_\new$ are all $G$-submodules and hence $W_\new$ can be seen to be a $G$-module in a natural way.  Otherwise, we must compute the action of $G$ on $W_\new$.

\begin{description}[leftmargin =0.5cm]
\item[Step 3.] For each subalgebra $B$ glued onto a set of axes $Y$, we extend the gluing as follows.  Since $U := W(Y) \subset W$ and $V_\new = W$, we now know all the products of elements in $U$, so we adjust the gluing.  Specifically, let $U_V = U \cap V$ and find a complement $D$ so that
\[
U = U_V \oplus D
\]
Then
\[
U_\new := U \oplus \mu(U_V, D) \otimes \mu(D, D)
\]
is the subalgebra of $W_\new$ generated by $Y$.  Note that $\mu(U_V, D) \cong U_V \otimes D$ and $\mu(D,D) \cong S^2(D)$.  We extend the map $\phi$ to $\phi_\new$ in the obvious way, by mapping the new products in $U_\new$ to the corresponding products in $B$.  Hence, $\phi_\new$ preserves multiplication.  Observe that $U_\new$ is also a $K$-submodule and so the homomorphism $\psi$ is unchanged.  Hence, $(\phi_\new, \psi)$ is a gluing of $B$ onto $Y$ in $W_\new$.
\end{description}

\begin{description}[leftmargin =0.5cm]
\item[Step 4.] For each gluing, we add the kernel of $\phi_\new$ to the space of relations.
\end{description}

Indeed, if $\phi_\new(v) = 0$, then $v$ must be the zero vector in any final axial algebra, hence it is a relation.

\begin{description}[leftmargin =0.5cm]
\item[Step 5.] For each axis $a$ and subalgebra $U_\new$ which contains $a$, we use $\phi_\new$ to pull back the eigenspaces of $B \cap \phi_\new(U_\new)$ to add to the eigenspaces in $W_\new$.
\end{description}

Since we only consider axes and gluings up to $G$-orbit, we must be careful as one orbit of axes may split into several orbits when intersected with the subalgebra.

We note that the above expansion step can be made to work if we do not expand to the whole of $W$, but just to some $G$-submodule $U$ of $W$ which contains $V$.  That is, we choose some subspace complement $C$ to $V$ in $U$ (picking it to be a $G$-submodule if possible) and we expand to
\[
W_\new = W \oplus V\otimes C \oplus S^2(C)
\]
and have $V_\new = U$.  The gluing for the subalgebras and the eigenspaces are updated similarly to above.  This partial expansion has the advantage that it is easier to do computationally as it is smaller and we may still be able to find relations.

\subsection*{Stage 2: Building up eigenspaces}

We begin by recovering the grading on $W_\new$, before finding further eigenvectors and relations. Recall that relations are simply elements 
of the eigenspace $W_{\new,\emptyset}$.

\begin{description}[leftmargin =0.5cm]
\item[Step 1.] For each axis $a$, we compute the action of $T_a$ on $W_\new$ and hence find the decomposition $W_\new = \bigoplus_{t \in T} W_{\new, t}$ 
with respect to $a$.
\end{description}

For example, in the Monster fusion law case, we have the $\mathbb{Z}_2$-decomposition $W_\new  = W_{\new, +} \oplus W_{\new, -}$, where $W_{\new,+}$ and $W_{\new,-}$ are the $1$- and $-1$-eigenspaces of $\tau_a$, respectively.

If $C$ is a submodule, then the calculation can be simplified as follows
\[
W_{\new, t} = W_t \oplus \bigoplus_{s \in T} (V_s \otimes C_{s^{-1}t}) \oplus \bigoplus_{s \in T} C_s \times C_{s^{-1}t}
\]
where $V_s$ and $C_s$ are the $T$-graded parts of $V$ and $C$ respectively.

We no longer need the old $W$, so we now drop the subscript and write $W$ for $W_\new$ and similarly $V$ for $V_\new$.

\begin{description}[leftmargin =0.5cm]
\item[Step 2.] We repeatedly apply the following techniques until the pure eigenspaces $W_I$ (including the relation eigenspace $W_{\emptyset}$) stop growing.
\begin{enumerate}
\item For each $t \in T$, we sum together and take intersections of the $W_I$ for each pure subset $I \subsetneqq \mathcal{F}_t$ as per Lemma \ref{sumup&int}.
\item For each $t \in T$, let $\lambda \in I \subseteq \mathcal{F}_t$.  For each $u \in W_I \cap V$, we add $ua - \lambda u$ to $W_{I - \lambda}$ as per Lemma \ref{multiplydown}.
\item We apply each useful fusion rule $I \star J = K$.  That is, for all $u \in W_I \cap V$ and $v \in W_J \cap V$, we add their product $uv$ to $W_K$.
\end{enumerate}
\end{description}

Note that in parts (2) and (3), we of course may just do these for a basis of the eigenspaces concerned.

In the case of the Monster fusion law, $\mathcal{F}_- = \{ \frac{1}{32} \}$.  So, for the odd subspace $W_-$, there are no subspaces to sum or intersect in part (1) above.  Also in part (2) for $W_-$, since the only choice for $\lambda$ is $\frac{1}{32}$, we obtain that $ua - \frac{1}{32}u\in W_{\emptyset}$ is a relation. Since $W_- = W_\frac{1}{32}$ will not grow in size, we need only apply part (2) once.  Also, as noted after Lemma \ref{Monsteruseful}, all the useful fusion rules for the Monster fusion law come from the even part.  Therefore, for the Monster fusion law, we only need apply part (2) once to the odd part and then just work on the even part.

\begin{description}[leftmargin =0.5cm]
\item[Step 3. (Optional)] If additionally we want to force that $G_a$ fixes every vector in $W_{1_T}$ (as is true for primitive algebras), then we may apply the technique from Lemma \ref{wh-w} to get $u^g - u \in W_{1_T-1}$ for all $g \in G_a$ and $u \in W_{1_T}$.
\end{description}

By the assumptions in Lemma \ref{wh-w}, we may only apply this lemma to subsets such that $1 \in I$.  We claim that it is enough to just apply it to $1_T$.  By the discussion at the beginning of the section, since $1 \in 1_T$ we need just consider pure subsets $I \subset {\cal F}_{1_T}$ with $1\in I$.  Let $u \in W_I \subset W_{1_T}$.  So, the vector $v = u^g -u$ is found in both $W_{I - 1}$ and $W_{1_T -1}$.  Since the action of $g \in G_a$ preserves the eigenspaces, we know trivially that $v \in W_I$.  So, by intersecting as in Step 2 (1), we recover that $v \in W_{I - 1} = W_{1_T -1} \cap W_I$.  Moreover, once we have done the expansion step, we know the decomposition given by the $T$-grading and this does not change until the next expansion step.  Hence, we need only apply Step 3 once per expansion.

\subsection*{Stage 3: Reduction}

If we have found some relations for our algebra (i.e. $W_{\emptyset}\neq 0$), we may reduce our partial algebra $W$ by factoring out by the relations.  Let $R$ be the $G$-submodule generated by the $W_{\emptyset}$.  Before forming the quotient, we search for additional relations by using the two following techniques.

First, if $R$ intersects $V$ non-trivially, then we may multiply $R \cap V$ by elements of $V$.  Since elements $r \in R$ are relations and must become zero in the target algebra, so are $vr$, for all $r \in R \cap V$ and $v \in V$.  So we repeatedly multiply by elements of $V$ to grow $R$ until the dimension of $R$ stabilises.

Secondly, suppose that $R$ intersects a subspace $U = W(Y)$ where we have glued in a subalgebra $B$.  Let $(\phi, \psi)$ be the gluing map.  Then $R' := \phi(U \cap R)$ are relations in the subalgebra $B$.  Since we know the multiplication in $B$, we can use the first technique to multiply by elements of $B$ to grow $R'$ (this may include multiplying by elements we do not yet know how to multiply by in $W$, hence giving us extra information).  We then pull back $R'$ to $W$ using $\phi^{-1}$ to get additional relations.

\begin{description}[leftmargin =0.5cm]
\item[Step 1.] We use the above two techniques repeatedly, until we find no further relations.  Let $\pi \colon W \to W/R$ be the quotient map.  We define $W_\new$ as the image $\pi(W)$, $V_\new = \pi(V)$ and $\mu_\new$ is the map induced by $\mu$.
\end{description}

\begin{description}[leftmargin =0.5cm]
\item[Step 2.] For each gluing, we update both the subspace and the subalgebra by taking $U_\new = \pi(U)$ and $B_\new = B/\pi(U \cap R)$ and updating the gluing maps accordingly.
\end{description}

\begin{description}[leftmargin =0.5cm]
\item[Step 3.] We transfer the axes and eigenspaces $W_I$ to $W_{\new}$ by applying $\pi$.
\end{description}

Note that if $R$ contains any relations of the form $a-b$ for axes $a$ and $b$, then we have reduced the (potential) algebra to one generated by a smaller set of axes $X'$.  Hence we may exit the algorithm.

Now that we have described our algorithm, we shall prove Theorem \ref{algorithmthm}.

\begin{proof}[Proof of Theorem $\ref{algorithmthm}$.]
It is clear from the construction of the algorithm that $A$ is spanned by products of axes in $X$.  Since each axis is contained in its own $1$-eigenspace, they are idempotents.  At stage 2 we use Lemma \ref{multiplydown}, so each axis must be semisimple.  Also at stage 2 we impose the fusion law, therefore the multiplication must satisfy this and hence $A$ is an axial algebra for the required fusion law.  By construction, for each axis $a \in X$ and $\chi \in T^*$, $\tau_a(\chi)$ is the corresponding Miyamoto automorphism and hence $G_0$ is the Miyamoto group.

Observe that any axial algebra $B$ with the same axes, Miyamoto group, $\tau$-map and shape must satisfy the relations we have factored by in our algorithm.  If we do not use Lemma \ref{wh-w} in stage 2, then we have not factored by any other relations and so $B$ must be a quotient of $A$.
\end{proof}

In practice, for reasons of efficiency, we perform some of the steps above in a different order.  For example, we may perform the reduction step at any stage.  In particular, it may be computationally advantageous to reduce once we find enough relations as any further calculations will be performed in a smaller space and hence may be quicker.

\section{Results}\label{sec:results}

In Table \ref{tab:results}, we present some of the results that the implementation of our algorithm \cite{ParAxlAlg} in {\sc magma} \cite{magma} has found.  Our current implementation is restricted to a $\mathbb{Z}_2$-graded fusion law with one eigenvalue in the negative part and the examples given in the table are all for the Monster fusion law.  All the results here are also over $\mathbb{Q}$, although our implementation works over finite fields and even function fields.  Note that, although in our algorithm and implementation we do not require that the $\tau$-map be bijective, this is the case we concentrate on in the table as this is the situation considered by Seress \cite[Table 3]{seress}.

The columns in the table are
\begin{itemize}
\item Miyamoto group $G_0$.
\item Axes, where we give the size decomposed into the sum of orbit lengths.
\item Shape.  Here we omit shapes of type $5\textrm{A}$ and $6\textrm{A}$ as where these occur they are uniquely defined.  If an algebra contains a $4\textrm{A}$, or $4\textrm{B}$, we omit to mention the $2\textrm{B}$, or $2\textrm{A}$, respectively, that is contained in it.  Likewise, we omit the $2\A$ and $3\A$ that are contained in a $6\A$.
\item Dimension of the algebra.  A question mark indicates that our algorithm did not complete and a $0$ indicates that the algebra collapses.
\item The minimal $m$ for which $A$ is $m$-closed.  Recall that an axial algebra is $m$-closed if it is spanned by products of length at most $m$ in the axes.
\item Whether the algebra has a $G_0$-invariant Frobenius form that is non-zero on the set of axes $X$.  If it is additionally positive definite or positive semi-definite, we mark this with a pos, or semi, respectively.
\end{itemize}

In addition to the results in the table, we have computed many of the smaller groups acting on larger numbers of axes. For example, we have computed $S_4$ acting on $6$, $6+6$, $6+6+6$, $12$, $12+12$, $12+12+12$, $1+3+6$, $1+3+6+6$, $1+3+3+6+6$, $3+6$, $3+3+6$ and $3+6+6$ axes, but we do not present these results here.  Several of these are useful for gluing in to complete examples for larger groups $G_0 \geq S_4$.

Compared to Seress \cite{seress}, we find several new algebras.  This includes several new examples that are $3$-closed, only one of which was previously known.  It also includes many examples that do not satisfy the M8 condition, or the $2\mathrm{Aa}$, $3\A$, or $4\A$ conditions, but nevertheless lead to examples.  Note that Seress considers both $A_6$ and the non-split extension $3^{\textstyle\cdot} A_6$.  However, $3^{\textstyle\cdot} A_6$ does not have a faithful transitive action on $45$ points with an admissible $\tau$-map.  Indeed, its only actions on $45$ points with an admissible $\tau$-map have kernel $C_3$ and $A_6$ acting faithfully.  So, there is no axial algebra with Miyamoto group $3^{\textstyle\cdot} A_6$ acting on $45$ axes.

We now note some interesting results coming from the computed examples:  In all the cases below, there is at most one class of admissible $\tau$-map, however this is not true in general.  For example, the group $2^4$ acting on $2+2+2+2$ axes has four classes of admissible $\tau$-maps at least three of which lead to non-trivial axial algebras.  All the examples found so far are primitive (although in most cases the optional step 3 in stage 2 using Lemma \ref{wh-w} was used to construct them).

The largest $m$ for which we have examples which are $m$-closed but not $(m-1)$-closed is $5$.  There are two such examples which are $S_4$ acting on $1+3+6$ axes with shapes $4\A3\A2\A2\B2\B$ and $4\A3\C2\A2\B2\B$.  These have dimension $52$ and $27$, respectively.

All the examples computed have a $G_0$-invariant Frobenius form that is non-zero on the axes and all these forms are positive semi-definite.  Although the vast majority are positive definite, there are examples for which the form is positive semi-definite but not positive definite.  For example there is an algebra for the group $V_4$ acting on $2+2+1$ axes with shape $4\A2\A2\A$ and it has dimension $14$.  The radical of the form is $3$-dimensional which gives an ideal in the algebra.  Once we factor out by this, the resulting algebra also has the same group, orbit structure of axes and shape and is primitive of dimension $11$ with a positive definite Frobenius form.

We have also found several different examples which do not satisfy the $2\mathrm{Aa}$, $3\A$, or $4\A$ conditions.  In particular, when $\tau$ is not bijective, or is bijective and $1_G$ is in the image of the $\tau$-map (so there is an isolated axis) it is easy to find such examples.  However, we should not expect these conditions to hold even when $\tau$ is a bijection.  The example for $A_6$ on $45$ axes of shape $4\B3\A3\C$ has dimension $105$, but does not satisfy the $3\A$ axiom.  There are pairs of $3\A$ subalgebras which are disjoint, but have the same induced Miyamoto group.

Finally, consider the example which we cannot complete for $S_3\times S_3$ with $3+3$ axes and shape $3\A3\A2\A$.  An algebra of this shape can be found in the algebra $A$ of shape $3\A2\A$ on $15$ axes.  Namely, if we consider the subalgebra generated by the $3+3$ axes this has the required shape.  Moreover, this subalgebra is in fact the full algebra $A$, but it is $4$-closed with respect to these $3+3$ axes.  Since $A$ is a quotient of the algebra $\hat{A}$ we are trying to compute, the algebra $\hat{A}$ of shape $3\A3\A2\A$ is at least $4$-closed, which may be one reason it is hard to construct even though it is a small group.

\begin{longtable}{cccccc}
$G_0$ & axes & shape & dim & $m$ & form\\
\hline
$S_3\times S_3$ & 3+3 & 3A3A2A & ? \\
$S_3\times S_3$ & 3+3 & 3A3A2B & 8 & 2 & pos\\
$S_3\times S_3$ & 3+3 & 3A3C2A & 0 & 0 & -\\
$S_3\times S_3$ & 3+3 & 3A3C2B & 7 & 2 & pos\\
$S_3\times S_3$ & 3+3 & 3C3C2A & 0 & 0 & -\\
$S_3\times S_3$ & 3+3 & 3C3C2B & 6 & 1 & pos\\
$S_3\times S_3$ & 3+9 & 3A3A & 18 & 2 & pos\\
$S_3\times S_3$ & 3+9 & 3A3C & 0 & 0 & -\\
$S_3\times S_3$ & 3+9 & 3C3A & 0 & 0 & -\\
$S_3\times S_3$ & 3+9 & 3C3C & 0 & 0 & -\\
$S_3\times S_3$ & 3+3+9 & 3A2A & 18 & 2 & pos\\
$S_3\times S_3$ & 3+3+9 & 3A2B & 25 & 3 & pos\\
$S_3\times S_3$ & 3+3+9 & 3C2A & 0 & 0 & -\\
$S_3\times S_3$ & 3+3+9 & 3C2B & 0 & 0 & -\\

&&&&\\
$S_4$ & 6 & 3A2A & 13 & 2 & pos \\
$S_4$ & 6 & 3A2B & 13 & 3 & pos\\
$S_4$ & 6 & 3C2A & 9 &  2 & pos \\
$S_4$ & 6 & 3C2B & 6 & 1 & pos \\
$S_4$ & 6+3 & 4A3A2A & 23 & 3 & pos\\
$S_4$ & 6+3 & 4A3A2B & 25 & 3 & pos\\
$S_4$ & 6+3 & 4A3C2A & 0 & 0 & -\\
$S_4$ & 6+3 & 4A3C2B & 12 &  2 & pos \\
$S_4$ & 6+3 & 4B3A2A & 13 &  2 & pos \\
$S_4$ & 6+3 & 4B3A2B & 16 &  2 & pos \\
$S_4$ & 6+3 & 4B3C2A & 9 &  1 & pos \\
$S_4$ & 6+3 & 4B3C2B & 12 &  2 & pos \\
&&&&\\
$A_5$ & 15 & 3A2A & 26 & 2 & pos\\
$A_5$ & 15 & 3A2B & 46 & 3 & pos\\
$A_5$ & 15 & 3C2A & 20 & 2 & pos\\
$A_5$ & 15 & 3C2B & 21 & 2 & pos\\
&&&&\\
$S_5$ & 10 & 3A2A & ? &  \\
$S_5$ & 10 & 3A2B & ? &  \\
$S_5$ & 10 & 3C2A & 0  &  0 & - \\
$S_5$ & 10 & 3C2B & 10 & 1 & pos\\
$S_5$ & 10+15 & 4A & 61 & 2 & pos\\
$S_5$ & 10+15 & 4B & 36 & 2 & pos\\
&&&&\\
$L_3(2)$ & 21 & 4A3A & ? &  \\
$L_3(2)$ & 21 & 4A3C & 57 & 3 & pos\\
$L_3(2)$ & 21 & 4B3A & 49 & 2 & pos\\
$L_3(2)$ & 21 & 4B3C & 21 & 1 & pos\\
&&&&\\
$A_6$ & 45 & 4A3A3A & ? & \\
$A_6$ & 45 & 4A3A3C & 0 & 0& -\\
$A_6$ & 45 & 4A3C3C & 187 & 3 & pos\\
$A_6$ & 45 & 4B3A3A & 76 & 2 & pos\\
$A_6$ & 45 & 4B3A3C & 105 & 2 & pos\\
$A_6$ & 45 & 4B3C3C & 70 & 2 & pos\\
&&&&\\
$S_6$ & 15 & 3A2A & ? & \\
$S_6$ & 15 & 3A2B & ? & \\
$S_6$ & 15 & 3C2A & 0 & 0 & -\\
$S_6$ & 15 & 3C2B & 15 & 1 & pos\\
$S_6$ & 15+15 & 4A3A3A2A & ? & \\
$S_6$ & 15+15 & 4A3A3A2B & ? & \\
$S_6$ & 15+15 & 4A3A3C2A & 0 & 0 & -\\
$S_6$ & 15+15 & 4A3A3C2B & 0 & 0 & -\\
$S_6$ & 15+15 & 4A3C3C2A & 0 & 0 & -\\
$S_6$ & 15+15 & 4A3C3C2B & 0 & 0 & -\\
$S_6$ & 15+15 & 4B3A3A2A & 0 & 0 & -\\
$S_6$ & 15+15 & 4B3A3A2B & ? & \\
$S_6$ & 15+15 & 4B3A3C2A & 0 & 0 & -\\
$S_6$ & 15+15 & 4B3A3C2B & 0 & 0 & -\\
$S_6$ & 15+15 & 4B3C3C2A & 0 & 0 & -\\
$S_6$ & 15+15 & 4B3C3C2B & 0 & 0 & -\\
$S_6$ & 15+45 & 4A4A3A2A & 151 & 2 & pos\\
$S_6$ & 15+45 & 4A4A3A2B & 0 & 0 & -\\
$S_6$ & 15+45 & 4A4A3C2A & 0 & 0 & -\\
$S_6$ & 15+45 & 4A4A3C2B & 0 & 0 & -\\
$S_6$ & 15+45 & 4B4B3A2A & 0 & 0 & -\\
$S_6$ & 15+45 & 4B4B3A2B & 91 & 2 & pos\\
$S_6$ & 15+45 & 4B4B3C2A & 0 & 0 & -\\
$S_6$ & 15+45 & 4B4B3C2B & 0 & 0 & -\\
$S_6$ & 15+15+45 & 4A2A2A2A & 151 & 2 & pos\\
$S_6$ & 15+15+45 & 4A2A2A2B & 0 & 0 & -\\
$S_6$ & 15+15+45 & 4A2A2B2B & 0 & 0 & -\\
$S_6$ & 15+15+45 & 4A2B2A2A & 151 & 2 & pos\\
$S_6$ & 15+15+45 & 4A2B2A2B & 0 & 0 & -\\
$S_6$ & 15+15+45 & 4A2B2B2B & 0 & 0 & -\\
$S_6$ & 15+15+45 & 4B2A2A2A & 0 & 0 & -\\
$S_6$ & 15+15+45 & 4B2A2A2B & 0 & 0 & -\\
$S_6$ & 15+15+45 & 4B2A2B2B & 0 & 0 & -\\
$S_6$ & 15+15+45 & 4B2B2A2A & 0 & 0 & -\\
$S_6$ & 15+15+45 & 4B2B2A2B & 0 & 0 & -\\
$S_6$ & 15+15+45 & 4B2B2B2B & 106 & 2 & pos\\
&&&&\\
$3.S_6$ & 45 & 3A & ?  \\
$3.S_6$ & 45 & 3C & 0 & 0 & - \\
$3.S_6$ & 45+45 & 3A2A & 0 & 0 & - \\
$3.S_6$ & 45+45 & 3A2B & 0 & 0 & - \\
$3.S_6$ & 45+45 & 3C2A & 0 & 0 & - \\
$3.S_6$ & 45+45 & 3C2B & 136 & 2 & pos\\
&&&&\\
$(S_4\times S_3) \cap A_7$ & 18 & 3A3A3A & ? \\
$(S_4\times S_3) \cap A_7$ & 18 & 3A3A3C & 0 & 0 & - \\
$(S_4\times S_3) \cap A_7$ & 18 & 3A3C3C & ? \\
$(S_4\times S_3) \cap A_7$ & 18 & 3C3C3C & ? \\
$(S_4\times S_3) \cap A_7$ & 18+3 & 4B3A3A3A2A & ? \\
$(S_4\times S_3) \cap A_7$ & 18+3 & 4B3A3A3A2B & ? \\
$(S_4\times S_3) \cap A_7$ & 18+3 & 4B3A3A3C2A & 0 & 0 & - \\
$(S_4\times S_3) \cap A_7$ & 18+3 & 4B3A3A3C2B & 0 & 0 & - \\
$(S_4\times S_3) \cap A_7$ & 18+3 & 4B3A3C3C2A & ? \\
$(S_4\times S_3) \cap A_7$ & 18+3 & 4B3A3C3C2B & ? \\
$(S_4\times S_3) \cap A_7$ & 18+3 & 4B3C3C3C2A & 24 & 2 & pos \\
$(S_4\times S_3) \cap A_7$ & 18+3 & 4B3C3C3C2B & 27 & 2 & pos \\
&&&&\\
$L_2(11)$ & 55 & 6A5A5A & 101 & 2 & pos\\
&&&&\\
$L_3(3)$ & 117 & 3A & 0 & 0 & -\\
$L_3(3)$ & 117 & 3C & 144 & 2 & pos\\
&&&&\\
$(S_5\times S_3) \cap A_8$ & 30 & 3A3A & ? \\
$(S_5\times S_3) \cap A_8$ & 30 & 3A3C & 0 & 0 & -\\
$(S_5\times S_3) \cap A_8$ & 30 & 3C3A & 0 & 0 & -\\
$(S_5\times S_3) \cap A_8$ & 30 & 3C3C & 0 & 0 & - \\
$(S_5\times S_3) \cap A_8$ & 15+30 & 3A & 67 & 2 & pos\\
$(S_5\times S_3) \cap A_8$ & 15+30 & 3C & 0 & 0 & - \\
\hline
\caption{Results}\label{tab:results}
\end{longtable}


\begin{thebibliography}{99}
\bibitem{B86} R.E. Borcherds, Vertex algebras, Kac-Moody algebras, and the monster, \textit{Proc. Natl. Acad. Sci. USA}, {\bf 83} (1986), 3068--3071.

\bibitem{magma} W. Bosma, J. Cannon, and C. Playoust, The Magma algebra system. I. The user language, \textit{J. Symbolic Comput.}, {\bf 24} (1997), 235--265.

\bibitem{C85} J.H. Conway, A simple construction for the Fischer-Griess monster group, \textit{Invent. Math.} {\bf 79} (1985), no. 3, 513--540. 

\bibitem{FLM} I. Frenkel, J. Lepowsky and A Meurman, \textit{Vertex operator algebras and the Monster}, Academic Press, Boston, MA, Pure and Applied Mathematics {\bf 134} (1988).

\bibitem{G82} R.L. Griess. The friendly giant. \textit{Invent. Math.} {\bf 69} (1982), 1--102.

\bibitem{Axial1} J.I. Hall, F. Rehren and S. Shpectorov, Universal axial algebras and a theorem of Sakuma, \textit{J. Algebra} {\bf 421} (2015), 394--424.

\bibitem{Axial2} J.I. Hall, F. Rehren and S. Shpectorov, Primitive axial algebras of Jordan type, \textit{J. Algebra} {\bf 437} (2015), 79--115.

\bibitem{I09} A.A. Ivanov, \emph{The Monster Group and Majorana Involutions}, Cambridge Univ. Press, Cambridge, Cambridge Tracts in Mathematics {\bf 176} (2009).

\bibitem{IPSS} A.A. Ivanov, D. V. Pasechnik, \'A Seress and S. Shpectorov, Majorana representations of the symmetric group of degree $4$, \textit{J. Algebra} {\bf 324} (2010), no. 9, 2432--2463. 

\bibitem{axialstructure} S.M.S. Khasraw, J. M\textsuperscript{c}Inroy and S. Shpectorov, On the structure of axial algebras, \textit{arXiv}:1809.10132, 27 pages, Sep 2018.

\bibitem{ParAxlAlg} J. M\textsuperscript{c}Inroy and S. Shpectorov, Partial axial algebras -- a magma package, https://github.com/JustMaths/AxialAlgebras.

\bibitem{Miy96} M. Miyamoto, Griess algebras and conformal vectors in vertex operator algebras, \textit{J. Algebra}, {\bf 179} (1996), no. 2, 523--548. 

\bibitem{Maddycode} M. Pfeiffer and M. Whybrow, Constructing Majorana Representations, \textit{arXiv}:1803.10723, 19 pages, Mar 2018.

\bibitem{felix} F. Rehren, Generalised dihedral subalgebras from the Monster, \textit{Trans. Amer. Math. Soc.} {\bf 369} (2017), no. 10, 6953--6986.

\bibitem{sakuma} S. Sakuma, $6$-transposition property of $\tau$-involutions of vertex operator algebras, \textit{Int. Math. Res. Not. IMRN}, no. 9 (2007). 

\bibitem{seress} \'A Seress, Construction of 2-closed M-representations, \textit{ISSAC $2012$--Proceedings of the $37$th International Symposium on Symbolic and Algebraic Computation}, ACM, New York, 2012, 311--318.
\end{thebibliography}
\end{document}